\documentclass[12pt]{article}

\usepackage{amssymb}
\usepackage{amsmath}
\usepackage{amsthm}
 \usepackage{dsfont}

  \newtheorem{theorem}{Theorem}
 \newtheorem{definition}{Definition}
  \newtheorem{corollary}{Corollary}
        
    \newtheorem{proposition}{Proposition}

\begin{document}

\title{Highly singular (frequentially sparse) steady solutions for the 2D Navier--Stokes equations on the torus}
\author{Pierre Gilles Lemari\'e-Rieusset\footnote{LaMME, Univ Evry, CNRS, Universit\'e Paris-Saclay, 91025, Evry, France; e-mail : pierregilles.lemarierieusset@univ-evry.fr}}
\date{}\maketitle

\begin{abstract} We construct non-trivial steady solutions in $H^{-1}$ for the 2D Navier--Stokes equations on the torus. In particular, the solutions are not square integrable, so that we have to redefine the notion of solutions.
\end{abstract}



\section*{Introduction}
In this paper, we are looking for steady solutions $\vec u$ of the 2D Navier--Stokes equations on the torus $\mathbb{T}^d=\mathbb{R}^2/ 2\pi\mathbb{Z}^2$, i.e.  for solutions of the equations
\begin{equation}\label{steady}\left\{ \begin{split}  \Delta\vec u -\mathbb{P} (\vec u\cdot \vec\nabla\vec u)=0\\ \textrm{ div } \vec u=0 
\end{split}\right.\end{equation} where $\vec u$ is a periodical distribution vector field, with mean value $0$:\[\int_{\mathbb{T}^2} \vec u(x)\, dx=0.\]

Such a periodical distribution vector field $\vec u$ can be written as a Fourier series
$$ \vec u(x)=\sum_{k\in\mathbb{Z}^2_+} \cos(k\cdot x)\vec v_k + \sin(k\cdot x) \vec w_k$$  where $k\in \mathbb{Z}^2_+$ if and only if $k=(k_1,k_2)\in \mathbb{Z}^2\setminus\{(0,0)\}$ and $\arg(k_1 +\mathrm{i} k_2)\in (-\pi/2,\pi/2]$.  The convergence in $\mathcal{D}'$ is given by a slow growth of the coefficients: 
$$ \vert \vec v_k\vert +\vert \vec w_k\vert \leq C \vert k\vert^N$$ for some constants $C$ and $N$; in particular, $\vec u\in H^s(\mathbb{T}^2)$ for $s<-N-1$.\\ If $k\in \mathbb{Z}^2$, $ k\neq (0,0)$ and $k\notin \mathbb{Z}^2_+$, then we may replace $k$ with $-k$ (with $-k\in \mathbb{Z}^2_+$) and write $\cos(k\cdot x)\vec v_k + \sin(k\cdot x) \vec w_k=\cos((-k)\cdot x)\vec v_k +\sin((-k)\cdot x)(- \vec w_k)$, hence the condition $k\in \mathbb{Z}^2_+$ is not essential.

 In our equations, $\mathbb{P}$ is the Leray projection operator on solenoidal vector fields, defined by
\begin{equation}\mathbb{P}(\sum_{k\in\mathbb{Z}^2_+} \cos(k\cdot x)\vec v_k + \sin(k\cdot x) \vec w_k)=\sum_{k\in\mathbb{Z}^2_+} \rho_k \cos(k\cdot x+\theta_k) k^\perp\end{equation} with $$  \rho_k \cos(k\cdot x+\theta_k)=  \cos(k\cdot x) \frac{\vec v_k\cdot k^\perp}{\vert k\vert^2}   + \sin(k\cdot x) \frac{\vec w_k\cdot k^\perp}{\vert k\vert^2} $$ where $$(k_1,k_2)^\perp=(-k_2,k_1), \vert (k_1,k_2)\vert^2=k_1^2+k_2^2 \textrm{ and } \rho_k=\frac{\sqrt{(\vec v_k\cdot k^\perp)^2+(\vec w_k\cdot k^\perp)^2} }{\vert k\vert^2}.$$
 
 It is easy to check that,   if the solution $\vec u$ satisfies $\vec u\in H^1(\mathbb{T}^2)$, then $\vec u=0$. Thus, we shall not require (weak) diffentiability for $\vec u$. Usually, it is customary to rewrite
  $\vec u\cdot\vec\nabla \vec u$ in the equations as  $\partial_1(u_1\vec u)+\partial_2(u_2\vec u)$  (since ${\rm div }\vec u=0$), where the derivatives are taken in the sense of distributions. In order to define $ u_i\vec u$, one then usually requires that $\vec u$ be square integrable. 
  
  As we whall see it, it is easy to check that, if the solution $\vec u$ satisfies $\vec u\in L^p(\mathbb{T}^2)$ for some $p>2$, then $\vec u=0$. This is even true when $\vec u$ belongs to the Lorentz space  $L^{2,1}(\mathbb{T}^2)$. This is still probably the case for $\vec u\in L^2(\mathbb{T}^2)$, thus we will search for some solution $\vec u$ which is not square integrable. We need however to be able to define $\mathbb{P}(\vec u\cdot\vec\nabla \vec u)$ when $\vec u$ is no longer square integrable.
  
\begin{definition}[Admissible vector fields] A divergence free periodical distribution vector field is admissible for the steady problem on $\mathbb{T}^2$ if it is of   the form
\begin{equation}\vec u=\sum_{k\in\mathbb{Z}^2_+} \rho_k \cos(k\cdot x+\theta_k) k^\perp\end{equation} with
\begin{equation} \!   \! \!  \! \!  \! \sum_{k_1\in\mathbb{Z}^2_+} \!  \sum_{k_2\in\mathbb{Z}^2_+}  \!  \rho_{k_1}\rho_{k_2}   \|\mathbb{P}\left(\cos(k_1 \! \cdot  \! x+\theta_{k_1})\sin (k_2 \! \cdot  \! x+\theta_{k_2}) (k_1^\perp \! \cdot  \! k_2) k_2^\perp\right)\|_{H^{-N}} \! < \! +\infty \end{equation}
for some $N$.
\end{definition}

If $\vec u= \sum_{k\in\mathbb{Z}^2_+} \rho_k \cos(k\cdot x+\theta_k) k^\perp=\sum_{k\in\mathbb{Z}^2_+} \vec u_k$ is an admissible vector field, we will then define $\mathbb{P}(\vec u\cdot\vec\nabla \vec u)$ as
\begin{equation}  \mathbb{P}(\vec u\cdot\vec\nabla \vec u)=\sum_{k_1\in\mathbb{Z}^2_+} \sum_{k_2\in\mathbb{Z}^2_+}   \mathbb{P}(\vec u_{k_1}\cdot\vec\nabla \vec u_{k_2}).\end{equation}
The main result in this paper is then the following one:

\begin{theorem} \label{main} There exists non-trivial solutions to the equations  \begin{equation}\left\{ \begin{split}  \Delta\vec u -\mathbb{P} (\vec u\cdot \vec\nabla\vec u)=0\\ \mathrm{ div } \vec u=0 
\end{split}\right.\end{equation} where $\vec u$ is an admissible vector field (with mean value $0$) with $\vec u\in H^{-1}(\mathbb{T}^2)\cap BMO^{-1}$.
\end{theorem}

The paper is organized in the following manner: in Section 1, we recall classical results on steady solutions on $\mathbb{T}^d$; in section 2, we describe some examples of admissible vector fields; in section 3, we prove Theorem \ref{main}; in section 4, we comment on the Koch and Tataru theorem.

\section{Steady solutions for the Navier-Stokes problem on $\mathbb{T}^d$: known results.}

In this section, we recall known results on  steady solutions  for the  Navier--Stokes problem in $L^2(\mathbb{T}^d)$, for $ d\geq 2$.

\subsection*{Case $\vec u\in H^1(\mathbb{T}^d)\cap L^4(\mathbb{T}^d)$.}
If  $\vec u\in H^1(\mathbb{T}^d)\cap L^4(\mathbb{T}^d)$ (recall that, when $d\leq 4$, $H^1(\mathbb{T}^d)\subset L^4(\mathbb{T}^d)$), we can compute $$ -\int_{\mathbb{T}^d}\vert \vec \nabla\otimes\vec u\vert^2\, dx=\int_{\mathbb{T}^d} \vec u\cdot \mathbb{P} (\vec u\cdot \nabla\vec u)\, dx=\int_{\mathbb{T}^d} \vec u\cdot   (\vec u\cdot \nabla\vec u)\, dx=\int_{\mathbb{T}^d} \mathrm{div}\, (\frac{\vert \vec u\vert^2}2\vec u) \, dx=0.
$$ Thus $\vec u=0$ (we are interested in vector fields with null mean value).

\subsection*{Case $\vec u\in L^p(\mathbb{T}^d)$, $p>d$.}
As $p$ is subcritical when $p>d$, the initial value problem 
\begin{equation}\label{evolut}\left\{ \begin{split} &\partial_t\vec v=\Delta \vec v-\mathbb{P}(\vec v\cdot\nabla\vec v)\\ & {\rm div }\,\vec v=0\\ & \vec v(0,x)=\vec u_0(x)\end{split}\right.\end{equation} with $\vec u_0\in L^p$ has a unique solution in $\mathcal{C}([0,T),L^p)$ for some time $T$. This solution $\vec v$ is smooth on $(0,T)\times\mathbb{T}^d$.

If $\vec u$ is a steady solution of (\ref{steady}), then $\vec v(t,x)=\vec u(x)$ defines a solution of the Cauchy problem (\ref{evolut}) with initial value $\vec u_0=\vec u$. Hence, if $\vec u\in L^p$, we find that $\vec u\in H^1\cap L^\infty$ and finally $\vec u=0$.

\subsection*{Case $\vec u\in L^d(\mathbb{T}^d)$, $d\geq 3$.}
The value $p=d$ is critical for the Cauchy problem (\ref{evolut}). When $\vec u_0\in L^d$, the problem has  a   solution $\vec v$  in $\mathcal{C}([0,T),L^d)$ for some time $T$ and this solution $\vec v$ is smooth on $(0,T)\times\mathbb{T}^d$. But uniqueness of solutions in $\mathcal{C}([0,T),L^d)$ is known only for $d\geq 3$. Thus, if $\vec u$ is a steady solution of (\ref{steady}) with $\vec u\in L^d(\mathbb{T}^d)$ and if $d\geq 3$, then $\vec u=0$.

\subsection*{Case $\vec u\in L^2(\mathbb{T}^d)$, $d\geq 4$.}

Recently, Luo \cite{LUO} constructed non trivial steady solutions in  $L^2(\mathbb{T}^d)$, $d\geq 4$ (this solution belongs to $L^p(\mathbb{T}^d)$ for some $p\in (2,d)$). His proof was following the scheme of convex integration developed by  De Lellis and Sz\'ekelyhidi \cite{DEL} in the case of non-steady solutions for the Euler equations  and by Buckmaster and Vicol \cite{BUC}  in the case of non-steady solutions for the  Navier--Stokes equations. However, his proof requires the spatial dimension $d$ to be no lesser than $4$.

\subsection*{Case $\vec u\in L^2(\mathbb{T}^2)$.}

 Uniqueness of solutions of    the Cauchy problem (\ref{evolut}) in $\mathcal{C}([0,T),L^2(\mathbb{T}^2))$    is not known. Proofs of uniqueness in $\mathcal{C}([0,T),L^3(\mathbb{T}^3))$ are based on maximal regularity properties which are no longer true in the 2D case \cite{FURa, LIO, MEY, MON, FURb, MAY}. On the other hand, the proof of non-uniqueness in  $\mathcal{C}([0,T),L^2(\mathbb{T}^3))$ is based on convex integration methods which cannot be applied  in the 2D case \cite{BUC}.
 
 However, we have uniqueness    in $\mathcal{C}([0,T),L^{2,1}(\mathbb{T}^2))$, where $L^{2,1}$ is a Lorentz space:
 
 \begin{proposition}$\ $\\ If $\vec v_1$ and $\vec v_2$ are two solutions of   the Cauchy problem (\ref{evolut}) with $\vec v_1, \vec v_2\in \mathcal{C}([0,T),L^{2,1}(\mathbb{T}^2))$ and $\vec v_1(0,.)=\vec v_2(0,.)=\vec u_0$, then $\vec v_1=\vec v_2$. 
 \end{proposition}
 
 \begin{proof} We follow the lines of \cite{FURb} and \cite{MEY}. If $$T^*=\sup\{S\geq 0\ /\ \vec v_1=\vec v_2 \text{ on } [0,S)\}$$ and if $T^*<T$, then $\vec v_1(T^*, .)=\vec v_2(T^*,.)$: it is obvious if $T^*=0$, and is a consequence of continuity if $0<T^*<T$. Moreover, we can write the integral formulation of the Navier--Stokes equations with initial time  $T^*$: for $T^*\leq t<T$ and $j=1, 2$
 \[ \vec v_j(t,.)=e^{(t-T^*)\Delta}\vec v_j(T^*,.) +\int_{T^*}^t\int_{\mathbb{R}^2} K(t-s,.-y) (\vec v_j(s,y)\otimes \vec v_j(s,y)\, dy\, ds\]
 with 
 \[ \vert K(t,x)\vert\leq C \frac 1{(\sqrt t+\vert x\vert)^3}\leq C (\mathds{1}_{\vert x\vert< \pi} \frac 1{t^{3/2}}+\mathds{1}_{\vert x\vert \geq \pi}\frac 1{\vert x\vert^3}).\] Let \[K_{\rm per} (t,x)= \sum_{k\in\mathbb{Z}^2} K(t,x-2\pi k).\] Then $K_{\rm per}(t,.)\in L^1(\mathbb{T}^2)\cap L^\infty (\mathbb{T}^2)$, with  \[ \| K_{\rm per}(t,.)\|_1\leq C \frac 1{\sqrt t} \text{ and  }\|K_{\rm per}(t,.)\|_\infty\leq C (1+\frac 1{t^{3/2}}).\] We write $\vec v_j=\vec w_0+\vec w_j$, with $\vec w_0=e^{(t-T^*)\Delta} \vec v_j(T^*,.)$. By density of $L^\infty(\mathbb{T}^2)$ in $L^{2,1}(\mathbb{T}^2)$, we have 
 \[ \lim_{\delta\rightarrow 0^+} \sup_{T^*<t<T^*+\delta} \sqrt{t-T^*} \|\vec w_0(t,.)\|_\infty=0,\] while, by continuity of $\vec v_j$ and $\vec w_0$, we have, for $j=1,2$,
  \[ \lim_{\delta\rightarrow 0^+}  \sup_{T^*<t<T^*+\delta} \|\vec w_j(t,.)\|_{L^{2,1}(\mathbb{T}^2)} =0.\]
  We now write, for $\vec w=\vec v_1-\vec v_2$ and $T^*\leq t<T$,
  \begin{equation*}\begin{split} \vec w(t,x)=&\int_{T^*}^t\int_{\mathbb{T}^2} K_{\rm per}(t-s,x-y) (\vec w(s,y)\otimes \vec w_0(s,y)) \, dy\, ds\\&+\int_{T^*}^t\int_{\mathbb{T}^2} K_{\rm per}(t-s,x-y) (\vec w(s,y)\otimes \vec  w_1(s,y)) \, dy\, ds
  \\&+\int_{T^*}^t\int_{\mathbb{T}^2} K_{\rm per}(t-s,x-y) (\vec w_0(s,y)\otimes \vec w(s,y)) \, dy\, ds
  \\&+\int_{T^*}^t\int_{\mathbb{T}^2} K_{\rm per}(t-s,x-y) (\vec w_2(s,y)\otimes \vec w(s,y)) \, dy\, ds\\=& \vec z_1(t,x)+\vec z_2(t,x)+\vec z_3(t,x)+\vec z_4(t,x)
  \end{split}\end{equation*} and we estimate $\|\vec w(t,.)\|_{L^{2,\infty}(\mathbb{T}^2)}$. 
  
  We first write, for $\vec w=\vec v_1-\vec v_2$ and $T^*\leq t<T$,
  \begin{equation*}\begin{split}  \|\vec z_1(t,.)\|_{L^{2,\infty}}+ \|\vec z_3(t,.)\|_{L^{2,\infty}}\leq & C \int_{T^*}^t  \|K_{\rm per}(t-s,.)\|_1 \|\vec w_0(s,.)\|_\infty \|\vec w(s,.)\|_{L^{2,\infty}}\, ds \\ \leq  C'\int_{T^*}^t\frac 1{\sqrt{t-s}\sqrt {s-T^*}}\,  ds & \sup_{T^*<s<t} \sqrt{s-T^*}\|\vec w_0(s,.)\|_\infty  \sup_{T^*<s<t} \|\vec w(s,.)\|_{L^{2,\infty}} \\ =  & \pi C'  \sup_{T^*<s<t} \sqrt{s-T^*}\|\vec w_0(s,.)\|_\infty  \sup_{T^*<s<t} \|\vec w(s,.)\|_{L^{2,\infty}}.
    \end{split}\end{equation*} 
    
    For $A>0$, we write
      \begin{equation*}\begin{split} \vec z_2(t,x)+\vec z_4(t,x)=&\int_{\sup(t-A,T^*)}^t\int_{\mathbb{T}^2} K_{\rm per}(t-s,x-y) (\vec w(s,y)\otimes \vec w_1(s,y)) \, dy\, ds\\&+\int_{T^*}^{\sup(t-A,T^*)}\int_{\mathbb{T}^2} K_{\rm per}(t-s,x-y) (\vec w(s,y)\otimes \vec  w_1(s,y)) \, dy\, ds
  \\&+\int_{\sup(t-A,T^*)}^t\int_{\mathbb{T}^2} K_{\rm per}(t-s,x-y) (\vec w_2(s,y)\otimes \vec w(s,y)) \, dy\, ds
  \\&+\int_{T^*}^{\sup(t-A,T^*)}\int_{\mathbb{T}^2} K_{\rm per}(t-s,x-y) (\vec w_2(s,y)\otimes \vec w(s,y)) \, dy\, ds\\=& \vec z_{5,A}(t,x)+\vec z_{6,A}(t,x)+\vec z_{7,A}(t,x)+\vec z_{8,A}(t,x).
  \end{split}\end{equation*} Since the pointwise product is bounded from $L^{2,1}\times L^{2,\infty}$ to $L^1$, we have
   \begin{equation*}\begin{split}  \|\vec z_{5,A}(t,.)\|_1\leq & C \int_{\sup(t-A,T^*)}^t  \|K_{\rm per}(t-s,.)\|_1 \|\vec w_1(s,.)\|_{L^{2,1}} \|\vec w(s,.)\|_{L^{2,\infty}}\, ds \\ \leq & C'\int_{\sup(t-A,T^*)}^t\frac 1{\sqrt{t-s}}\,  ds  \sup_{T^*<s<t} \|\vec w_1(s,.)\|_{L^{2,1}}  \sup_{T^*<s<t} \|\vec w(s,.)\|_{L^{2,\infty}} \\ \leq   & 2 C'  \sqrt A  \sup_{T^*<s<t}  \|\vec w_1(s,.)\|_{L^{2,1}}  \sup_{T^*<s<t} \|\vec w(s,.)\|_{L^{2,\infty}}.
    \end{split}\end{equation*} 
    Similarly
    \[ \|\vec z_{7,A}(t,.)\|_1\leq  C \sqrt A  \sup_{T^*<s<t}  \|\vec w_2(s,.)\|_{L^{2,1}}  \sup_{T^*<s<t} \|\vec w(s,.)\|_{L^{2,\infty}}. \]
    On the other hand, we have (for $T^*\leq t<\min(T,T^*+1)$)
      \begin{equation*}\begin{split}  \|\vec z_{6,A}(t,.)\|_\infty\leq & C \int_{T^*}^{\sup(t-A,T^*)}  \|K_{\rm per}(t-s,.)\|_\infty \|\vec w_1(s,.)\|_{L^{2,1}} \|\vec w(s,.)\|_{L^{2,\infty}}\, ds \\ \leq & C'\int_{T^*}^{\sup(t-A,T^*)}\frac 1{(t-s)^{3/2}}\,  ds  \sup_{T^*<s<t} \|\vec w_1(s,.)\|_{L^{2,1}}  \sup_{T^*<s<t} \|\vec w(s,.)\|_{L^{2,\infty}} \\ \leq   & 2 C'  \frac 1{\sqrt A}  \sup_{T^*<s<t}  \|\vec w_1(s,.)\|_{L^{2,1}}  \sup_{T^*<s<t} \|\vec w(s,.)\|_{L^{2,\infty}}.
    \end{split}\end{equation*} 
    Similarly
    \[ \|\vec z_{8,A}(t,.)\|_\infty\leq  C \frac 1{ \sqrt A}  \sup_{T^*<s<t}  \|\vec w_2(s,.)\|_{L^{2,1}}  \sup_{T^*<s<t} \|\vec w(s,.)\|_{L^{2,\infty}}. \]
 As $L^{2,\infty}=[L^1,L^\infty]_{\frac 1 2,\infty}$, we find that, for $T^*\leq t<\min(T,T^*+1)$,
      \[ \|\vec z_2(t,.)\|_{L^{2,\infty}} \leq   C  \sup_{T^*<s<t}  \|\vec w_1(s,.)\|_{L^{2,1}}   \sup_{T^*<s<t} \|\vec w(s,.)\|_{L^{2,\infty}} \] and       \[ \|\vec z_4(t,.)\|_{L^{2,\infty}}\leq   C  \sup_{T^*<s<t}  \|\vec w_2(s,.)\|_{L^{2,1}}  \sup_{T^*<s<t} \|\vec w(s,.)\|_{L^{2,\infty}}. \]
      
      Putting together  those estimates, we get that, for $0<\delta<\min(1,T-T^*)$,
      \[ \sup_{T^*\leq t\leq T^*+\delta} \|\vec w(t,.)\|_{L^{2,\infty}} \leq C A(\delta) \sup_{T^*\leq t\leq T^*+\delta} \|\vec w(t,.)\|_{L^{2,\infty}} \] with
      \[ A(\delta)=\sup_{T^*<t<T^*+\delta} \sqrt{t-T^*} \|\vec w_0(t,.)\|_\infty+  \|\vec w_1(t,.)\|_{L^{2,1}(\mathbb{T}^2)} +   \|\vec w_2(t,.)\|_{L^{2,1}(\mathbb{T}^2)}. \]
As 
 \[ \lim_{\delta\rightarrow 0^+}A(\delta)=0,\] we get that $\vec w=\vec v_1-\vec v_2$ is equal to $0$ on $[0, T^*+\delta]$ for $\delta$ small enough, in contradiction with the definition of $T^*$. Thus $T^*=T$, and $\vec v_1=\vec v_2$.
 \end{proof}

 \begin{corollary}$\ $\\ If $\vec u$ is a steady solution of (\ref{steady}) with $\vec u\in L^{2,1}(\mathbb{T}^2)$, then $\vec u=0$.
 \end{corollary}
 \begin{proof} We consider   the Cauchy problem (\ref{evolut}) where the initial  value $\vec u_0$ is equal to our steady solution $\vec u$. We can construct a mild solution $\vec v_1$ on a small time interval $[0,T]$ such that $\vec v_1\in\mathcal{C}([0,T], L^{2,1})$, $\sup_{0<t<T} \sqrt t \|\vec v_1(t,.)\|_\infty<+\infty$ and $\lim_{t\rightarrow 0^+} \sqrt t \|\vec v_1(t,.)\|_\infty=0$. We have another solution in $\mathcal{C}([0,T], L^{2,1})$, namely $\vec v_2(t,.)=\vec u$. By uniqueness, we find that $\vec u=\vec v_1(\frac T 2,.)\in L^\infty$, and thus $\vec u=0$.
 \end{proof}

\section{Admissible vector fields.}

In this section, we describe some examples of admissible divergence free periodical distribution vector fields 
\begin{equation}\vec u=\sum_{k\in\mathbb{Z}^2_+} \rho_k \cos(k\cdot x+\theta_k) k^\perp=\sum_{k\in\mathbb{Z}^2_+} \vec u_k.\end{equation} 
\subsection*{Square integrable vector fields}
The most obvious example is the case $\vec u\in L^2$, i.e. $\sum_{k\in\mathbb{Z}^2_+} \|\vec u_k\|_2^2<+\infty$. We have $\vec u_{k_1}\otimes \vec u_{k_2}\in L^1$ while the frequencies appearing in $\vec u_{k_1}\otimes \vec u_{k_2}$ are $k_1+k_2$ and $k_1-k_2$ (if $k_1\neq k_2$, since $\vec u_k\cdot\vec\nabla \vec u_k=0$). Thus,
$$ \|\mathbb{P}(\vec u_{k_1}\cdot \vec\nabla \vec u_{k_2})\|_{H^{-N}}\leq  C_N \|\vec u_{k_1}\|_2  \|\vec u_{k_2}\|_2 (\vert k_1+k_2\vert^{-N-1} +\vert k_1-k_2\vert^{-N-1} ).$$ If $N>1$, we have $$\sum_{j\in\mathbb{Z}^2\setminus \{(0,0)\} } \vert j\vert^{-N-1}<+\infty,$$ hence 
$$ \sum_{j\in\mathbb{Z}^2\setminus \{(0,0)\} }  \vert j\vert^{-N-1} \|\vec u_{k_2+j}\|_2\in l^2$$ and thus
$$\sum_{k_1\in\mathbb{Z}^2_+} \sum_{k_2\in\mathbb{Z}^2_+}  \|\mathbb{P}(\vec u_{k_1}\cdot \vec\nabla \vec u_{k_2})\|_{H^{-N}}<+\infty.$$
\subsection*{Lacunary Fourier series}
Let us consider a lacunary Fourier series
$$ \vec u=\sum_{j=0}^{+\infty}  \rho_{k_j} \cos(k_j\cdot x+\theta_{k_j}) k_j^\perp=\sum_{j=0}^{+\infty} \vec u_{k_j}$$ with $$\vert k_{j+1}\vert > 8 \vert k_j\vert$$ and $$\sum_{j=0}^{+\infty} \rho_{k_j}^2 \vert k_j\vert^{-2N}<+\infty$$ where $
N\geq 0$ (so that $\vec u\in H^{-N-1}$). We have. 
 \begin{equation*}\begin{split} \|\mathbb{P}(\vec u_{k_j}\cdot \vec\nabla \vec u_{k_p})\|_{H^{-2N-3}} \leq& C \rho_{k_j}\rho_{k_p} \vert k_j\vert \vert k_p\vert (\vert k_j\vert+\vert k_p\vert)^{-2N-2}\\\leq& C' \vert k_j\vert^{-N} \rho^{k_j}\vert k_p\vert^{-N} \rho^{k_p} \left(\frac {\min(\vert k_j\vert, \vert k_p\vert)}{ \max(\vert k_j\vert, \vert k_p\vert)}\right)^{N+1}.\end{split}\end{equation*} Noticing that
 $$ \sum_{j=1}^{+\infty} \sum_{0\leq p\leq j-1} \rho_{k_j}\rho_{k_p} \frac{\vert k_j\vert \vert k_p\vert }{(\vert k_j\vert+\vert k_p\vert)^{2N+2}}\leq  C  \sum_{j=1}^{+\infty} \sum_{0\leq p\leq j-1} \frac{ \rho_{k_j}}{\vert k_j\vert^{N}}\frac{\rho_{k_p}}{\vert k_p\vert^{N}} 8^{-(j-p)(N+1)} ,$$ we find that 
 $$\sum_{p=0}^{+\infty} \sum_{j=0}^{+\infty} \| \mathbb{P}(\vec u_{k_p}\cdot\vec\nabla \vec u_{k_j})\|_{H^{-2N-3}}<+\infty.$$\\
 
 \noindent\textbf{Remark: } We could have proved that $\mathbb{P}(\vec u \cdot\vec \nabla\vec u)\in H^{-2N-3}$ in another way: we have $\vec u\in H^{-N-1}\cap B^{-N-1}_{\infty,\infty}$; using paradifferential calculus and decomposing the product $\vec u \cdot\vec \nabla\vec u$ in two paraproducts and a remainder, we see that the paraproducts are controlled in $H^{-2N-3}$ by $\|\vec u\|_{H^{-N-1}} \|\vec u\|_{B^{-N-1}_{\infty,\infty}}$, while the remainder is equal to $0$.
\subsection*{Lacunary resonant Fourier modes}
$$ \vec u=\sum_{j=0}^{+\infty}  \rho_{k_j} (\cos(k_j\cdot x+\theta_{k_j}) k_j^\perp+\cos((k_j+\omega_j)\cdot x+\eta_{k_j}) (k_j+\omega_j)^\perp)=\sum_{j=0}^{+\infty} \vec u_{k_j}$$ with $$\vert k_{j+1}\vert > 8 \vert k_j\vert,\  \vert k_j\vert> 8 \vert \omega_j\vert,\  \omega_j\cdot k_j=0$$ and $$\sum_{j=0}^{+\infty} \rho_{k_j}^2 \frac{\vert k_j\vert}{\vert\omega_j\vert}<+\infty.$$ 
We write $$\vec v_{k_j}=  \rho_{k_j} \cos(k_j\cdot x+\theta_{k_j}) k_j^\perp \text{ and }\vec w_{k_j}= \rho_{k_j} \cos((k_j+\omega_j)\cdot x+\eta_{k_j}) (k_j+\omega_j)^\perp.$$

In particular, we have that $\sum_{j=0}^{+\infty} \rho_{k_j}^2 <+\infty$, so that $\vec u\in H^{-1}\cap B^{-1}_{\infty,\infty}$. Following the computations of the case of lacunary solutions, we find that 
 $$\sum_{p=0}^{+\infty} \sum_{0\leq j, j\neq p} \| \mathbb{P}(\vec v_{k_p}\cdot\vec\nabla \vec v_{k_j})\|_{H^{-3}}<+\infty, \sum_{p=0}^{+\infty} \sum_{0\leq j, j\neq p} \| \mathbb{P}(\vec v_{k_p}\cdot\vec\nabla \vec w_{k_j})\|_{H^{-3}}<+\infty,$$  $$\sum_{p=0}^{+\infty} \sum_{0\leq j, j\neq p} \| \mathbb{P}(\vec w_{k_p}\cdot\vec\nabla \vec v_{k_j})\|_{H^{-3}}<+\infty, \sum_{p=0}^{+\infty} \sum_{0\leq j, j\neq p} \| \mathbb{P}(\vec w_{k_p}\cdot\vec\nabla \vec w_{k_j})\|_{H^{-3}}<+\infty.$$

 We now estimate the diagonal terms $ \mathbb{P}(\vec u_{k_j}\cdot\vec\nabla \vec u_{k_j}).$ We have 
$$\vec v_{k_j}\cdot\vec\nabla \vec v_{k_j}=\vec w_{k_j}\cdot\vec\nabla \vec w_{k_j}=0,$$ while
\begin{equation*}\begin{split} \vec v_{k_j}\cdot \vec \nabla \vec w_{k_j}=&\rho_{k_j}^2 (k_j^\perp\cdot \omega_j) \cos(k_j\cdot x+\theta_{k_j})  \cos((k_j+\omega_j)\cdot x+\eta_{k_j}+\frac\pi 2) (k_j+\omega_j)^\perp\\=&\frac 1 2 \rho_{k_j}^2 (k_j^\perp\cdot \omega_j) \cos((2k_j+\omega_j)\cdot x+\theta_{k_j}+\eta_{k_j}+\frac \pi 2)  (k_j+\omega_j)^\perp \\&+\frac 1 2 \rho_{k_j}^2 (k_j^\perp\cdot \omega_j) \cos(\omega_j\cdot x-\theta_{k_j}+\eta_{k_j}-\frac \pi 2)  (k_j+\omega_j)^\perp
\end{split}\end{equation*} and \begin{equation*}\begin{split} \vec w_{k_j}\cdot \vec \nabla \vec v_{k_j}=&\rho_{k_j}^2 (k_j\cdot \omega_j^\perp) \cos(k_j\cdot x+\theta_{k_j}+\frac \pi 2)  \cos((k_j+\omega_j)\cdot x+\eta_{k_j}) k_j^\perp\\=&\frac 1 2 \rho_{k_j}^2 (k_j\cdot \omega_j^\perp) \cos((2k_j+\omega_j)\cdot x+\theta_{k_j}+\eta_{k_j}+\frac \pi 2)  k_j^\perp \\&+\frac 1 2 \rho_{k_j}^2 (k_j\cdot \omega_j^\perp) \cos(\omega_j\cdot x-\theta_{k_j}+\eta_{k_j}-\frac \pi 2)  k_j^\perp.
\end{split}\end{equation*}
We have
$$\| \mathbb{P}( \rho_{k_j}^2 (k_j^\perp\cdot \omega_j) \cos((2k_j+\omega_j)\cdot x+\theta_{k_j}+\eta_{k_j}+\frac \pi 2)  (k_j+\omega_j)^\perp)\|_{H^{-3} } \leq C \rho_{k_j}^2 \frac{\vert \omega_j\vert}{\vert k_j\vert}$$
and
$$\| \mathbb{P}( \rho_{k_j}^2 (k_j\cdot \omega_j^\perp) \cos((2k_j+\omega_j)\cdot x+\theta_{k_j}+\eta_{k_j}-\frac \pi 2)  k_j^\perp)\|_{H^{-3} } \leq C \rho_{k_j}^2 \frac{\vert \omega_j\vert}{\vert k_j\vert}.$$
On the other hand, we have 
$$ \mathbb{P}( \rho_{k_j}^2 (k_j\cdot \omega_j^\perp) \cos(\omega_j\cdot x-\theta_{k_j}+\eta_{k_j}+\frac \pi 2)  k_j^\perp)=0$$
and
$$ \mathbb{P}( \rho_{k_j}^2 (k_j^\perp\cdot \omega_j) \cos(\omega_j\cdot x-\theta_{k_j}+\eta_{k_j}-\frac \pi 2)  (\omega_j+k_j)^\perp)=  \rho_{k_j}^2 (k_j^\perp\cdot \omega_j) \cos(\omega_j\cdot x-\theta_{k_j}+\eta_{k_j}-\frac \pi 2)  \omega_j^\perp$$
so that
$$ \|\mathbb{P}( \rho_{k_j}^2 (k_j^\perp\cdot \omega_j) \cos(\omega_j\cdot x-\theta_{k_j}+\eta_{k_j}-\frac \pi 2)  (\omega_j+k_j)^\perp)\|_{H^{-3}}\leq C \rho_{k_j}^2 \frac{\vert k_j\vert}{\vert \omega_j\vert}.$$
Thus, we get
 $$\sum_{j=0}^{+\infty}   \| \mathbb{P}(\vec v_{k_j}\cdot\vec\nabla \vec w_{k_j})\|_{H^{-3}}<+\infty, \sum_{j=0}^{+\infty}  \| \mathbb{P}(\vec w_{k_j}\cdot\vec\nabla \vec v_{k_j})\|_{H^{-3}}<+\infty,$$   and $\mathbb{P}(\vec u\cdot \vec\nabla \vec u)$  is well defined in $H^{-3}$.
\section{2D steady solutions.}
We are going to prove Theorem \ref{main} following the lines of \cite{BUC} and \cite{LUO}, i.e. applying the convex integration scheme by using intermittencies in the Fourier spectrum of the solution.
 In our case, however, computations will be much more simple than in the ones in \cite{BUC} and \cite{LUO}, as we don't bother on convergence in $L^2$.

We shall look for  a solution $$\vec u=\sum_{j=0}^{+\infty} \vec u_{j}=\vec u_0+\sum_{j=1}^{+\infty} \vec v_{j}+\vec w_j$$ 
where
\begin{itemize}
\item[$\bullet$] $\vec u_0=\rho_0 \cos(k_0\cdot x) k_0^\perp$ with $0<\rho_0<1$ and $k_0\in \mathbb{Z}^2\setminus\{(0,0)\}$,
\item[$\bullet$] $\vec v_{j}=  \rho_{j} \cos(k_j\cdot x) k_j^\perp ${ and } $\vec w_{j}= \rho_{j} \cos((k_j+\omega_j)\cdot x+\eta_{j}) (k_j+\omega_j)^\perp$ with $\rho_j>0$ and $k_j$, $\omega_j\in \mathbb{Z}^2\setminus\{(0,0)\}$,
\item[$\bullet$] for $j\geq 1$, $\vert k_{j}\vert > 8 \vert k_{j-1}\vert$,  $\vert k_j\vert> 8 \vert \omega_j\vert$, $\omega_j\cdot k_j=0$.
\end{itemize}
$k_j$, $\omega_j$ and $\eta_j$ will be constructed by induction and we'll check that
$$\sum_{j=1}^{+\infty} \rho_{j}^2 \frac{\vert k_j\vert}{\vert\omega_j\vert}<+\infty,$$ so that
$\vec u$ is an admissible vector field such that $\vec u\in H^{-1}$ (hence $\Delta\vec u\in H^{-3}$) and $\mathbb{P}(\vec u\cdot \vec\nabla\vec u)\in H^{-3}$.

Defining $\vec U_n=\sum_{j=0}^n \vec u_j$, we have the convergence of $\Delta\vec U_n-\mathbb{P}(\vec U_n\cdot\vec\nabla \vec U_n)$ to $\Delta\vec u-\mathbb{P}(\vec u\cdot\vec\nabla\vec u)$ in $H^{-3}$. We write, for $n\geq 1$, $$\Delta\vec U_n-\mathbb{P}(\vec U_n\cdot\vec\nabla \vec U_n)=\vec V_0+\sum_{j=1}^n \vec V_j+\vec W_j$$  
\begin{itemize}
\item[$\bullet$] $\vec V_0=\Delta \vec u_0(=\Delta \vec u_0-\mathbb{P}(\vec u_0\cdot\vec\nabla\vec u_0))=-\rho_0 \vert k_0\vert^2 \cos(k_0\cdot x) k_0^\perp$
\item[$\bullet$] for $n\geq 1$, \begin{equation*}\begin{split}\vec V_{n}=  \Delta \vec U_n&-\mathbb{P}(\vec u_n\cdot\vec \nabla \vec U_{n-1}) -\mathbb{P}(\vec U_{n-1}\cdot\vec \nabla \vec u_n) 
\\&-\frac 1 2 (k_n^\perp\cdot\omega_n) \rho_n^2\mathbb{P}( \cos((2k_n+\omega_n)\cdot x+\eta_n+\frac\pi 2) (k_n+\omega_n)^\perp)
\\&-\frac 1 2 (k_n\cdot\omega_n^\perp) \rho_n^2\mathbb{P}( \cos((2k_n+\omega_n)\cdot x+\eta_n-\frac\pi 2) k_n^\perp)
 \end{split}\end{equation*}
\item[$\bullet$]  for $n\geq 1$, $\vec W_{n}= -\frac 1 2 \rho_{n}^2 (k_n^\perp\cdot \omega_n) \cos(\omega_n\cdot x+\eta_{n}-\frac \pi 2)  \omega_n^\perp$  
\end{itemize}

Let us write $A_n$ for the set of frequencies involved in the expansion of $\vec V_n$:
$$ \vec V_n=\sum_{k\in A_n}   \cos(k\cdot x+\alpha_{n,k}) \vec v_{n,k}=\sum_{k\in A_n} \lambda_{n,k} \cos(k\cdot x+\alpha_{n,k}) k^\perp,$$ with $\lambda_{n,k}= \frac{ \vec v_{n,k}\cdot k^\perp} {\vert k\vert^2}$.
Using the formula
 \begin{equation*}\begin{split} \mathbb{P} (\cos(\alpha.\cdot x+\theta)&\alpha^\perp\cdot \vec\nabla (\cos(\beta\cdot x+\eta)\beta^\perp)+\cos(\beta\cdot x+\eta)\beta^\perp\cdot \vec\nabla( \cos(\alpha.\cdot x+\theta)\alpha^\perp))\\ =& -  \mathbb{P}  ((\alpha^\perp\cdot \beta) \cos(\alpha.\cdot x+\theta)\sin(\beta\cdot x+\eta) \beta^\perp)
 \\ &-  \mathbb{P} ((\beta^\perp\cdot \alpha) \sin(\alpha.\cdot x+\theta)\cos(\beta\cdot x+\eta)\alpha^\perp)
\\=&- \frac 1 2\mathbb{P}((\cos((\alpha+\beta) \cdot x+\theta+\eta-\frac \pi 2) ((\alpha^\perp\cdot \beta)\beta^\perp +(\beta^\perp\cdot\alpha)\alpha^\perp))
\\ &- \frac 1 2\mathbb{P}((\cos((\alpha-\beta) \cdot x+\theta-\eta-\frac \pi 2) (-(\alpha^\perp\cdot \beta)\beta^\perp +(\beta^\perp\cdot\alpha)\alpha^\perp))
\\=&-\frac 1 2 \cos((\alpha+\beta) \cdot x+\theta+\eta+\frac \pi 2)  (\beta^\perp\cdot\alpha) \frac{\vert \beta\vert^2-\vert\vec \alpha\vert^2} {\vert \alpha+\beta\vert^2}(\alpha+\beta)^\perp
\\ &- \frac 1 2 \cos((\alpha-\beta) \cdot x+\theta-\eta+\frac \pi 2)  (\alpha^\perp\cdot \beta)  \frac{\vert \beta\vert^2-\vert\vec \alpha\vert^2} {\vert \alpha+\beta\vert^2} (\alpha-\beta)^\perp\end{split}\end{equation*}
  we see that we have more precisely $8n-1$ frequencies in $A_n$ for $n\geq 1$:
  \begin{itemize}
  \item[$\bullet$] $k=k_n$ with $\lambda_{n,k}= -    \rho_n \vert k_n\vert^2 $ and $\eta_{n,k}=0$
  \item[$\bullet$] $k=k_n+\omega_n$ with $\lambda_{n,k}= -    \rho_n \vert k_n+\omega_n\vert^2 $ and $\eta_{n,k}=\eta_n$
  \item[$\bullet$] $k= 2k_n+\omega_n$ with $\lambda_{n,k}= -\frac 1 2 \frac{(k_n^\perp\cdot\omega_n)^2}{\vert 2 k_n+\omega_n\vert^2} \rho_n^2  $ and $\eta_{n,k}=\eta_n-\frac\pi 2 $
  \item[$\bullet$] for $j=0,\dots, n-1$, $k=k_n+k_j $ with $\lambda_{n,k}=   \frac 1 2 \rho_n\rho_j  (k_j^\perp\cdot k_n) \frac{\vert k_j\vert^2-\vert  k_n\vert^2} {\vert k_n+k_j\vert^2} $ and $\eta_{n,k}=\frac\pi 2 $
  \item[$\bullet$] for $j=0,\dots, n-1$, $k= k_n-k_j$ with $\lambda_{n,k}= \frac 1 2 \rho_n\rho_j    (k_n^\perp\cdot k_j)  \frac{\vert k_j\vert^2-\vert\vec k_n\vert^2} {\vert k_n-k_j\vert^2} $ and $\eta_{n,k}=\frac\pi 2 $
  \item[$\bullet$] for $j=1,\dots, n-1$, $k=k_n+k_j +\omega_j$ with $\eta_{n,k}=\eta_j+\frac\pi 2 $ and \\ $\lambda_{n,k}=   \frac 1 2 \rho_n\rho_j  ((k_j+\omega_j)^\perp\cdot k_n) \frac{\vert k_j+\omega_j\vert^2-\vert  k_n\vert^2} {\vert k_n+k_j+\omega_j\vert^2} $ 
  \item[$\bullet$] for $j=1,\dots, n-1$, $k=k_n-k_j -\omega_j$ with $\eta_{n,k}=-\eta_j+\frac\pi 2 $ and \\ $\lambda_{n,k}=   \frac 1 2 \rho_n\rho_j    (k_n^\perp\cdot( k_j+\omega_j))  \frac{\vert k_j+\omega_j\vert^2-\vert\vec k_n\vert^2} {\vert k_n-k_j-\omega_j\vert^2} $ 
   \item[$\bullet$] for $j=0,\dots, n-1$, $k=k_n+\omega_n+k_j $ with $\eta_{n,k}=\eta_n+\frac\pi 2 $ and  $\lambda_{n,k}=   \frac 1 2 \rho_n\rho_j  (k_j^\perp\cdot (k_n+\omega_n)) \frac{\vert k_j\vert^2-\vert  k_n+\omega_n\vert^2} {\vert k_n+\omega_n+k_j\vert^2} $  
  \item[$\bullet$] for $j=0,\dots, n-1$, $k= k_n+\omega_n-k_j$ with  $\eta_{n,k}=\eta_n+\frac\pi 2 $ and  $\lambda_{n,k}= \frac 1 2 \rho_n\rho_j    ((k_n+\omega_n)^\perp\cdot k_j)  \frac{\vert k_j\vert^2-\vert\vec k_n+\omega_n\vert^2} {\vert k_n+\omega_n-k_j\vert^2} $ 
  \item[$\bullet$] for $j=1,\dots, n-1$, $k=k_n+\omega_n+k_j +\omega_j$ with $\eta_{n,k}=\eta_n+\eta_j+\frac\pi 2 $ and \\ $\lambda_{n,k}=   \frac 1 2 \rho_n\rho_j  ((k_j+\omega_j)^\perp\cdot (k_n+\omega_n)) \frac{\vert k_j+\omega_j\vert^2-\vert  k_n+\omega_n\vert^2} {\vert k_n+\omega_n+k_j+\omega_j\vert^2} $ 
  \item[$\bullet$] for $j=1,\dots, n-1$, $k=k_n+\omega_n-k_j -\omega_j$ with $\eta_{n,k}=\eta_n+-\eta_j+\frac\pi 2 $ and  $\lambda_{n,k}=   \frac 1 2 \rho_n\rho_j    ((k_n+\omega_n)^\perp\cdot( k_j+\omega_j))  \frac{\vert k_j+\omega_j\vert^2-\vert\vec k_n+\omega_n\vert^2} {\vert k_n+\omega_n-k_j-\omega_j\vert^2} $ 
\end{itemize}  
   For $k\in A_n$, we find that $\frac 5 8\vert k_n\vert\leq \vert k\vert \leq \frac {11}8\vert k_n\vert $, with $$\frac {11}8\vert k_n\vert \leq \frac {11} {64}\vert k_{n+1}\vert \leq \frac 3 8 (\frac 5 8\vert k_{n+1}\vert)$$
   and the frequencies occuring in $A_{n+1}$ are greater than those occuring in $A_n$.
   
   We then write $$A_0=\{\gamma_1\}, A_1=\{\gamma_2,\dots,\gamma_8\},\dots , A_n=\{\gamma_{4n^2-5n+3},\dots,\gamma_{4n^2+3n+1}\}, \dots$$
   We write, for $j\geq 0$,  
   \[ \vec V_j=\sum_{\gamma_p\in A_j} \lambda_p \cos(\gamma_p\cdot x+\alpha_p) \gamma_p^\perp=\sum_{\gamma_p\in A_j} \vert \lambda_p \vert \cos(\gamma_p\cdot x+\alpha_p+\epsilon_p\pi) \gamma_p^\perp\] with $\epsilon_p\in\{0,1\}$. Thus we have
   \begin{equation*}\begin{split} \Delta\vec U_n-\mathbb{P}(\vec U_n\cdot\vec\nabla \vec U_n)=& \sum_{j=0}^n \sum_{\gamma_p\in A_j} \vert \lambda_p\vert \cos(\gamma_p\cdot x+\alpha_p+\epsilon_p\pi) \gamma_p^\perp
  \\&-\sum_{j=1}^n \frac 1 2 \rho_{j}^2 (k_j^\perp\cdot \omega_j) \cos(\omega_j\cdot x+\eta_{j}-\frac \pi 2)  \omega_j^\perp.\end{split}
   \end{equation*}
   We know the values of $\rho_0$, $k_0$, hence of $\gamma_1=k_0$,  $\vert\lambda_1\vert  =\rho_0 \vert k_0\vert^2$ and  $\alpha_0+\epsilon_0\pi=\pi$.
 We shall define  by induction   $\omega_n$ , $k_n$, $\rho_n$ and $\eta_n$ for $n\geq 1$: we remark that $\gamma_n\in A_{j(n)}$ for some $j(n)<n$ (as $n< 4n^2-5n+3$). Thus, if we already know $\omega_j$ , $k_j$, $\rho_j$ and $\eta_j$ for $0\leq j\leq n-1$, we already know $\gamma_n$, $\vert \lambda_n\vert$ and $\alpha_n+\epsilon_n\pi$. The main idea is then to require that
 \[  \vert \lambda_n\vert \cos(\gamma_n\cdot x+\alpha_n+\epsilon_n\pi) \gamma_n^\perp= \frac 1 2 \rho_{n}^2 (k_n^\perp\cdot \omega_n) \cos(\omega_n\cdot x+\eta_{n}-\frac \pi 2)  \omega_n^\perp.\]
We thus make the following choices:   \begin{itemize}
   \item[$\bullet$] We take $\omega_n=\gamma_n$.
 \item[$\bullet$]   We take $ k_n= N_n \omega_n^\perp$, where the integer $N_n$ will fulfill some requirements. Our first requirement  will be that  $N_n\in \mathbb{N}$ is  large enough to grant that $N_n>8$ and $\vert k_n\vert> 8\vert k_{n-1}\vert$.
  \item[$\bullet$]   We then have
    \[    \frac 1 2 \rho_{n}^2 (k_n^\perp\cdot \omega_n) \cos(\omega_n\cdot x+\eta_{n}-\frac \pi 2)  \omega_n^\perp.= \frac 1 2 \rho_{n}^2  N_n \vert\omega_n\vert^2 \cos(\omega_n\cdot x+\eta_{n}+\frac \pi 2)  \omega_n^\perp. \] Thus, we take
    \[  \rho_n=\sqrt{\frac{2\vert\lambda_n\vert}{N_n \vert\omega_n\vert^2}}\text{ and } \eta_n=\alpha_n+\epsilon_n\pi-\frac\pi 2.\]
    \item[$\bullet$]  We shall add another requirement on $N_n$ in order to grant that \[\sum_{j=1}^{+\infty} \rho_{j}^2 \frac{\vert k_j\vert}{\vert\omega_j\vert}<+\infty,\]   Recall that $0<\rho_0<1$. Take $N_0=1$. We first check by induction that $\rho_n\leq  \rho_0 N_n^{-\frac 1 4} (\leq 1)$. Indeed, there is a constant $C_0$ such that $ \vert \omega_n\vert \geq \frac1{C_0} \vert k_{j(n)}\vert$ and 
 $$ \vert \lambda_n\vert\leq C_0 \vert k_{j(n)}\vert^2 \rho_{j(n)} \sup(1, \rho_0,\dots,\rho_{j(n)})$$ so that, by induction, $\vert \lambda_n\vert  \leq  C_0 \vert k_{j(n)}\vert^2 \frac{\rho_0}{N_{j(n)}^{1/4}}$ and $\rho_n\leq \sqrt{\frac{2C_0^3}{N_n}}\leq \rho_0 N_n^{-1/4}$ (if we take $N_n\geq 4 C_0^6 \rho_0^{-4}$).
 We have
 $$\rho_n^2\frac{\vert k_n\vert}{\vert\omega_n\vert}=N_n \rho_n^2=\frac{2\vert \lambda_n\vert}{\vert\omega_n\vert^2}\leq 2  C_0^3\rho_0 N_{j(n)}^{-1/4} $$  Thus,
  $$\sum_{j=1}^{+\infty} \rho_{j}^2 \frac{\vert k_j\vert}{\vert\omega_j\vert}\leq 2  C_0^3\rho_0\sum_{n=0}^{+\infty} \sum_{j\in A_n}  N_n^{-1/4}\leq  2  C_0^3\rho_0(1+\sum_{n=1}^{+\infty} (8n-1) N_n^{-1/4}). $$ Hence, our last requirement on $N_n$ will be that $N_n\geq (8n-1)^{12}$. 
   \end{itemize} 
 
 Theorem \ref{main} is proved.

\section{A remark on the Koch--Tataru theorem.}
In our construction, we have $\vec u\in   H^{-1}\cap BMO^{-1}$ with $$ \|\vec u\|_{H^{-1}} +\|\vec u\|_{BMO^{-1}}< C \rho_0.$$ Moreover, 
$$ \|\vec u-\vec U_n\|_{BMO^{-1}}\leq  C \rho_0 j(n)^{-3}\rightarrow_{n\rightarrow +\infty}  0.$$  By the Koch--Tataru theorem \cite{KOC}, for $\rho_0$ small enough, the evolutionary problem
\begin{equation}\label{evol}\left\{ \begin{split} &\partial_t\vec v=\Delta \vec v-\mathbb{P}(\vec v\cdot\nabla\vec v)\\ & {\rm div }\,\vec v=0\\ & \vec v(0,x)=\vec u(x)\end{split}\right.\end{equation} will have a smooth solution on $(0,+\infty)\times \mathbb{T}^2$ such that
\begin{itemize}
\item[$\bullet$] $\sup_{t>0} \sqrt t \|\vec v(t,.)\|_\infty<+\infty$
\item[$\bullet$]  $\sup_{t>0}   t \|\vec\nabla\otimes \vec  v(t,.)\|_\infty<+\infty$
\item[$\bullet$]  $\vec v\in \mathcal{C}([0,+\infty), BMO^{-1})$.
\end{itemize}
The steady solution $\vec u$ is another solution of the evolutionary problem (\ref{evol}), with $\vec u\in \mathcal{C}([0,+\infty), BMO^{-1})$. Of course, $\vec v\neq \vec u$ as $\lim_{t\rightarrow +\infty} \|\vec v(t,.)\|_{H^{-1}}=0$.
 
\maketitle

\end{document}